\documentclass[11pt,a4paper]{amsart}
\usepackage{amscd}
\usepackage{amssymb}
\usepackage[centering,text={15.5cm,22cm}]{geometry}
\usepackage{eqnarray}

\usepackage{hyperref}

\newtheorem{Theorem}{Theorem}[section]
\newtheorem{Corollary}[Theorem]{Corollary}
\newtheorem{Conjecture}[Theorem]{Conjecture}
\newtheorem{remark}[Theorem]{Remark}
\newtheorem{definition}[Theorem]{Definition}
\newtheorem{example}[Theorem]{Example}
\newtheorem{lemma}[Theorem]{Lemma}
\newtheorem{proposition}[Theorem]{Proposition}

\numberwithin{equation}{section}

\newcommand{\ZZ} {\mathbb{Z}}

\newcommand{\CC} {\mathbb{C}}

\newcommand{\PP} {\mathbb{P}}

\newcommand {\cO}  {\mathcal{O}}

\newcommand {\shF}  {\mathcal{F}}

\newcommand {\shI}  {\mathcal{I}}

\newcommand {\shO}  {\mathcal{O}}

\newcommand {\Ext}  {\operatorname{Ext}}

\newcommand {\Hom}  {\operatorname{Hom}}

\DeclareMathOperator {\hhh}{H}

\newcommand{\IP}{\mathbb P}

\newcommand{\ptwo}{\IP^2}
\newcommand{\pone}{\IP^1}

\newcommand{\mc}{\mathcal}

\DeclareMathOperator{\euler}{e}

\DeclareMathOperator{\tot}{Tot}

\DeclareMathOperator{\pic}{Pic}

\begin{document}

\title
[Local BPS Invariants]
{Local BPS Invariants: Enumerative Aspects and Wall-Crossing}

\author[J. Choi]{Jinwon Choi}
\address{Sookmyung Women's University, Cheongpa-ro 47-gil 100, Youngsan-gu, Seoul 04310, Republic of Korea}
\email{jwchoi@sookmyung.ac.kr}

\author[M. van Garrel]{Michel van Garrel}
\address{Mathematics Institute, Zeeman Building, University of Warwick, Coventry CV4 7AL, UK}
\email{michel.van-garrel@warwick.ac.uk}

\author[S. Katz]{Sheldon Katz}
\address{Department of Mathematics, MC-382, University of Illinois at Urbana-Champaign, Urbana, IL 61801, USA}
\email{katz@math.uiuc.edu}

\author[N. Takahashi]{Nobuyoshi Takahashi}
\address{Department of Mathematics, Graduate School of Science, Hiroshima University, 1-3-1 Kagamiyama, Higashi-Hiroshima, 739-8526 JAPAN}
\email{tkhsnbys@hiroshima-u.ac.jp}

\begin{abstract}
We study the BPS invariants for local del Pezzo surfaces, which can be obtained as the signed Euler characteristic of the moduli spaces of stable one-dimensional sheaves on the surface $S$. We calculate the Poincar\'{e} polynomials of the moduli spaces for the curve classes $\beta$ having arithmetic genus at most 2. We formulate a conjecture that these Poincar\'{e} polynomials are divisible by the Poincar\'{e} polynomials of $((-K_S).\beta-1)$-dimensional projective space. This conjecture motivates upcoming work on log BPS numbers \cite{logbps}. 
\end{abstract}

\maketitle

\section{Introduction}

Given a Calabi-Yau threefold $X$, physical reasoning was used in \cite{GV1,GV2} to produce the Gopakumar-Vafa invariants from moduli spaces of one-dimensional sheaves on $X$.  There are mathematical definitions of these invariants $n^g_\beta\in\ZZ$ given in \cite{katz} for $g=0$ and (consistently)  in \cite{MT} for general $g$.  In this paper, we refer to these invariants as the \emph{BPS invariants} of $X$, conjectured to be related to other enumerative invariants of $X$ as described in \cite{GV2}.

In this paper, we let $X$ be a local del Pezzo surface, the total space of the canonical bundle $K_S$ of a del Pezzo surface $S$.  
In this case, stable one-dimensional sheaves on $X$ can be identified with stable one-dimensional sheaves on $S$.

While we include calculations of the refined BPS indices and the higher genus BPS invariants in Section~\ref{sec:refined}, our main focus in this paper is on the genus zero BPS invariants $n_\beta=n^0_\beta$. Henceforth, \emph{local BPS invariants} shall mean the genus zero BPS invariants, unless specified otherwise. 

The moduli space of one-dimensional stable sheaves $M_\beta$ of class $\beta\in H_2(S,\ZZ)$ on $X$ is equipped with a symmetric obstruction theory. It follows from \cite{Behrend} that $n_\beta=\deg [M_\beta]^{\rm vir}$. In the local del Pezzo surface case, $M_\beta$ is smooth and hence the degree of the virtual cycle is given by the signed topological Euler characteristic of $M_\beta$. In the present paper, we calculate the Betti numbers of $M_\beta$. The results of our calculations, stated at the level of the local BPS invariants, are as follows.

\begin{Theorem}[Theorem~\ref{thm:localbps}] \label{thm:localcalc}
Let $\beta$ be either a line class, a conic class or a nef and big curve class on a del Pezzo surface $S$ of arithmetic genus $p_a(\beta)$ at most 2. Let $w=(-K_S).\beta$ and let $\eta$ be the maximum number of disjoint lines $l$ such that $\beta. l =0$. Then we have 
\begin{itemize}
\item[(i)] if $p_a(\beta)=0$, then $n_\beta= (-1)^{w-1} w$,
\item[(ii)] if $p_a(\beta)=1$ and $\beta\ne -K_{S_8}$, then $n_\beta= (-1)^{w-1} w (e(S)-\eta)$,
\item[(iii)] if $\beta=-K_{S_8}$, then $n_\beta=12$,
\item[(iv)] if $p_a(\beta)=2$ and $\beta\ne -2K_{S_8}$, then $n_\beta= (-1)^{w-1} w \left(\binom{\euler(S)-\eta}{2}+5\right)$.
\end{itemize}
Here, $e(-)$ denotes the topological Euler characteristics. 
\end{Theorem}

The main observation is that $n_\beta$ is divisible by $w$. Moreover, if we denote by $P_t(M)$ the Poincar\'{e} polynomial of a variety $M$, our calculations suggest the following. 

\begin{Conjecture}[Conjecture~\ref{conj:div}]\label{conj:div1}
Let $\beta$ be either a line class, a conic class or a nef and big curve class on a del Pezzo surface $S$. Let $w=(-K_S).\beta$. Then the Poincar\'{e} polynomial $P_t(M_\beta)$ has a factor $P_t(\PP^{w-1})$ and the quotient $P_t(M_\beta)/P_t(\PP^{w-1})$ is a palindromic polynomial. Consequently, $n_\beta$ is divisible by $w$. 
\end{Conjecture}

When $S=\PP^2$, Conjecture~\ref{conj:div1} is shown to be true whenever $P_t(M_\beta)$ or $n_\beta$ is calculated. See for example \cite{cc, cc2} for calculations of $P_t(M_\beta)$ up to degree 6 and \cite[\S 8.3]{kkv} for a calculation of $n_\beta$ up to degree 10. In this paper, we prove that Conjecture~\ref{conj:div1} holds for all del Pezzo surfaces and $\beta$ with $p_a(\beta)\le 2$.

In \cite{tak compl, tak log ms}, it was observed that up to degree 8, the log BPS numbers for $\PP^2$ agree with counts of rational curves of given degree which intersect a fixed smooth elliptic curve $E$ on $\PP^2$ only at a given special point and are smooth at this point. In a sequel \cite{logbps}, we generalize this idea to give a rigorous direct definition for the log BPS numbers by using the log GW theory for the pair of a del Pezzo surface $S$ and a smooth anticanonical curve $E$ on $S$. 
In a different direction, \cite[Conjecture 44]{Bou} stipulates a relationship, after a change of variable, of $P_t(M_\beta)$ with a generating function of certain higher genus log Gromov-Witten invariants. Combining \cite{Bou} and \cite{logbps} suggests a reconstruction result of higher genus log Gromov-Witten invariants in terms of genus 0 invariants.

Our strategy to prove Theorem~\ref{thm:localcalc} is as follows. We use the wall-crossing in the moduli space of $\delta$-stable pairs. The same strategy is used in \cite{cc} to study $M_\beta$ when $S=\PP^2$. See \S~\ref{sec:pairs} for a review of $\delta$-stable pair theory. When $\delta$ is sufficiently large, the space of $\delta$-stable pairs is isomorphic to the space of pairs $(C,Z)$ of curves $C$ in class $\beta$ and 0-dimensional subschemes $Z\subset C$. The latter space is a projective bundle over a Hilbert scheme provided $\beta$ is sufficiently very ample (Proposition~\ref{prop:PT}). More precisely, in our main situation of pairs with holomorphic Euler characteristic $1$, the latter space is a projective bundle if $\beta$ is very ample (resp. base point free) when the arithmetic genus of $\beta$ is $2$ (resp. $1$).

It is known that $\beta$ is very ample (resp.\ base point free) 
if and only if $\beta$ has positive (resp.\ nonnegative) intersection 
with all lines (i.e.\ $(-1)$-curves) on $S$ and $\beta\not=-2K_{S_8}$ 
(resp.\ $\beta\not=-K_{S_8}$). 
We show that the moduli space $M_\beta$ remains unchanged under blowing down $S$ along a $(-1)$-curve and taking the pushforward of $\beta$. After blowing down all $(-1)$-curves $l$ with $\beta.l=0$, the moduli space of stable pairs can be computed and hence $M_\beta$ can also be computed through wall-crossing.

The rest of this paper is organized as follows. In \S~\ref{prelim}, we collect basic facts about curve classes on del Pezzo surfaces. In \S~\ref{sec:local}, we start by 
reviewing stability of one-dimensional sheaves and give a definition of local BPS invariants in \S~\ref{sec:bps}. We prove the blowup property of BPS invariants in \S~\ref{sec:blowup}. We review the theory of $\delta$-stable pairs and wall-crossing in \S~\ref{sec:pairs}. In \S\ref{sec:localcal}, we compute the Poincar\'{e} polynomial of $M_\beta$ and prove Theorem~\ref{thm:localcalc}. Throughout this paper, we work over $\CC$.

\addtocontents{toc}{\protect\setcounter{tocdepth}{0}}
\section*{Acknowledgements}
\addtocontents{toc}{\protect\setcounter{tocdepth}{1}}

We would like to thank Pierrick Bousseau, Kiryong Chung, David Eisenbud, Tom Graber, Mark Gross, Young-Hoon Kiem, Martijn Kool, Davesh Maulik, Rahul Pandharipande, Helge Ruddat, Bernd Siebert, Jan Stienstra and Richard Thomas for enlightening conversations on several aspects relating to this work. A good amount of the present work was done while varying subsets of the authors were at or visiting the Korea Institute for Advanced Study (KIAS). We thank KIAS for its hospitality and the excellent research environment. JC is supported by the Sookmyung Women's University Research Grants (1-1603-2039), Korea NRF grant NRF-2018R1C1B6005600 and NRF-2015R1C1A1A01054185. MvG is supported by the German Research Foundation DFG-RTG-1670 and the European Commission Research Executive Agency MSCA-IF-746554. SK is supported in part by NSF grant DMS-1502170 and NSF grant DMS-1802242, as well as by NSF grant DMS-1440140 while in residence at MSRI in Spring, 2018. NT is supported by JSPS KAKENHI Grant Number JP17K05204.

\section{Preliminaries}\label{prelim}
In this section, we collect basic facts about curve classes on del Pezzo surfaces. Let $S$ be a del Pezzo surface. Denote by $S_r$ the blowup of $\ptwo$ along $r$ general points. Then $S$ is either $S_r$ for $0\le r \le 8$ or $\pone\times\pone$. We will mainly consider the case $S=S_r$ and will make remarks for for $\pone\times\pone$ separately whenever needed. The results of this paper hold for $\pone\times\pone$ as well. 

\begin{definition}
A class $\beta\in \hhh_2(S,\ZZ)$ is a curve class if it can be represented by a nonempty subscheme of dimension one. We often consider $\beta$ as a divisor on $S$.
\end{definition}

Since del Pezzo surfaces are rational, by Poincar\'e duality, $\pic(S)\cong\hhh_2(S,\ZZ)$. So when we write $|\shO_S(\beta)|$ or simply $|\beta|$, we mean the complete linear system $|L|$ for the unique $L\in\pic(S)$ such that $c_1(L)=\beta$.

For $S_r$, let $h$ be the pullback of $\mathcal{O}_{\ptwo} (1)$ and let $e_i$ for $1\le i\le r$ be the exceptional divisors. The Picard group $\mathrm{Pic}(S_r)$ is generated by $h$ and the $e_i$'s. We use the notation $(d; a_1,\cdots ,a_r)$ for the divisor $dh-\sum a_i e_i$. When there are repetitions in the $a_i$'s, we sometimes use superscripts to indicate the number of repetitions. For example, $(1;1^2)$ means the class $h-e_1-e_2$. The anticanonical divisor is $-K_{S_r}=(3;1^r)$. 

For $\pone\times\pone$, we denote by $h_1$ and $h_2$ the pullback of $\mathcal{O}_{\pone} (1)$ from each factor. The anticanonical divisor is $-K_{\pone\times\pone}=2h_1+2h_2$.

\begin{definition}
  A \emph{line class} on $S$ is a class $l\in \mathrm{Pic}(S)$ such that $l^2=-1$ and $(- K_S).l=1$. 
\end{definition}

It is well-known that each line class contains a unique irreducible line and there are only finitely many lines on $S$. 
\begin{example}
  By numerical calculation, we list all line classes up to permutation of the $e_i$'s: 
  \[e_i, (1; 1^2) ,(2; 1^5),(3; 2, 1^6), (4; 2^3, 1^5) , (5; 2^6, 1^2), (6; 3, 2^7). \]
\end{example}

\begin{definition}
  A divisor $D$ on $S$ is said to be \emph{nef} if $D.C\ge 0$ for any curve $C$. A nef divisor $D$ is said to be \emph{big} if in addition $D^2>0$. 
\end{definition}

\begin{definition}
  A line bundle $L$ on $S$ is said to be \emph{$k$-very ample} for an integer $k\ge 0$ if given any 0-dimensional subscheme $Z$ of $S$ of length $k+1$, the restriction map $H^0(L)\to H^0(L|_Z)$ is surjective. A divisor $D$ is said to be \emph{$k$-very ample} if the associated line bundle is $k$-very ample. 
\end{definition}
Note that $0$-very ample divisors are globally generated divisors and $1$-very ample divisors are very ample divisors. Di Rocco in \cite{Rocco} found the following numerical criterion for $k$-very ampleness on a del Pezzo surface.

\begin{lemma}[{\cite{Rocco},\cite[(2.1.1)]{BS}}]\label{lem:Rocco}
Let $D\in \mathrm{Pic}(S)$ be a divisor and $k\ge 0$ be an integer. Suppose that $D\ne -kK_{S_8}$ and $D\ne -(k+1)K_{S_8}$ and that $D\ne -K_{S_7}$ when $k=1$. Then $D\in \mathrm{Pic}(S)$ is $k$-very ample if and only if
\begin{itemize}
\item[(i)] when $S=\PP^2$, $D.h \ge k$,
\item[(ii)] when $S=\pone\times\pone$, $D.h_i\ge k$ for $i=1,2$,   
\item[(iii)] when $S=S_1$, $D. l\ge k$ for any line class $l$ and $D.(h-e_1)\ge k$,
\item[(iv)] when $S=S_r$ for $r \ge 2$, $D. l\ge k$ for any line class $l$.
\end{itemize} 
\end{lemma}

\begin{lemma}[{\cite{Rocco}}]\label{lem:Rocconef}
An effective divisor $D\in \mathrm{Pic}(S)$ is nef if and only if it is $0$-very ample, except for the case $D=-K_{S_8}$, which is nef but not $0$-very ample. 
\end{lemma}

\begin{lemma}[{\cite[\S 2.3 (P5, P7)]{Knut}}]\label{lem:nefbig}
If $D$ is base point free, then $\hhh^i(D)=0$ for $i=1,2$. Furthermore, if $D$ is nef and big, then a general member of $|D|$ is smooth and irreducible. 
\end{lemma}

\begin{definition}
For $\beta\in  \mathrm{Pic}(S)$, we denote by $p_a(\beta)$ the arithmetic genus of $\beta$. By adjunction we have
\[ p_a(\beta)= \frac12 \beta(\beta+K_S) +1.\] 
\end{definition}

\begin{definition}
A \emph{conic class} on $S$ is a class $D\in \mathrm{Pic}(S)$ such that $p_a(D)=0$ and $(-K_S).D=2$.
\end{definition}

By the remark after Corollary 3.3 of \cite{tvv}, nef but non-big divisors on $S$ are multiples of conic classes. The complete linear system of a conic class $D$ has projective dimension one, which gives a ruling $S\to \pone$. The fiber class of this ruling is $D$. 

\begin{example} \label{exam:g0w2}
The list of all conic classes is obtained by numerical calculation as follows (up to permutations of the $e_i$'s). See also \cite[Appendix A]{kkv}.
\begin{align}
&(1; 1) , (2; 1^4), (3; 2, 1^5) , (4; 2^3, 1^4),(4; 3, 1^7), \notag \\ 
& (5; 2^6, 1), (5; 3, 2^3, 1^4), (6; 3^2, 2^4, 1^2) , (7; 3^4, 2^3, 1), (7; 4, 3, 2^6) ,\label{g0d2cl}\\
& (8; 3^7, 1),(8; 4, 3^4, 2^3), (9; 4^2, 3^5, 2),  (10; 4^4, 3^4),  (11; 4^7, 3). \notag
\end{align}
\end{example}

\section{Local BPS invariants and Pairs}\label{sec:local}
\subsection{BPS invariants}\label{sec:bps}
Let $X$ be a Calabi-Yau threefold, and fix an ample line bundle $L$ on $X$. The Hilbert polynomial of $F$ is defined by $\chi(F\otimes L^m)$. In case $X=\tot(K_S)$, we consider coherent sheaves $F$ on $X$ set-theoretically supported on $S$, so that we can consider its homology class in $\hhh_2(S, \mathbb{Z})$.
\begin{definition}\label{def:stable}
A sheaf $F$ supported on a curve of class $\beta$ is called \emph{stable} if 
\begin{itemize}
\item[(a)] $F$ is pure, i.e., $F$ has no zero dimensional subsheaves. 
\item[(b)] For any proper nonzero subsheaf $G$ of $F$, we have \[ \frac{\chi(G)}{r(G)} <\frac{\chi(F)}{r(F)},\]
where $r(F)$ is the linear coefficient of the Hilbert polynomial of $F$. 
\end{itemize} 
\end{definition}

\begin{definition}
We denote by $M_{\beta,n}$ the moduli space of stable sheaves $F$ on $S$ with $\chi(F)=n$ and $[F](:=c_1(F))=\beta$. When $n=1$, we simply write $M_\beta=M_{\beta,1}$. 
\end{definition}

When $X$ is a smooth projective variety, $M_\beta$ is projective. The moduli space $M_\beta$ carries a symmetric obstruction theory \cite{katz}, and hence a virtual invariant is well defined. It is known that this virtual invariant is independent of the choice of the ample line bundle $L$. See for example \cite[Lemma 4.8]{Toda}. For a del Pezzo surface $S$, we take $L=-K_S$. Note that when $\chi(F)=1$, a proper nonzero subsheaf $G$ of $F$ is destabilizing if and only if $\chi(G)\ge 1$. The following proposition is well-known (See \cite{LeP}).

\begin{proposition} \label{prop:bps equiv}\noindent
\begin{itemize}
\item[(i)] Provided it is nonempty, $M_\beta$ is smooth of dimension $\beta^2+1$. 
\item[(ii)] $n_\beta=(-1)^{\beta^2+1} e(M_\beta)$.
\end{itemize} 
\end{proposition}
In Proposition~\ref{prop:bps equiv} and below, $\beta^2$ denotes the self-intersection $\beta.\beta$.
\begin{proof}
The obstruction at $F\in M_\beta$ is given by $\Ext^2(F,F)$. By Serre duality, we have 
\[
\Ext^2(F,F)= \Hom(F, F\otimes K_S)^\vee.
\]
The latter space is zero because $F$ is stable with respect to $-K_S$(\cite[Proposition 1.2.7]{HL}). Therefore $M_\beta$ is smooth. Moreover by Riemann-Roch, 
\[ \chi(F,F)=1-\mathrm{ext}^1(F,F)= \int_S \mathrm{ch}^\vee(F)\mathrm{ch}(F) \mathrm{td}(S)=-\beta^2.\]
The dimension of $M_\beta$ at $F$ is $\mathrm{ext}^1(F,F)=\beta^2+1$.  

Let $X$ be the total space of $K_S$ and let $M_\beta(X)$ be the moduli space of stable sheaves on $X$ with the same numerical condition as $M_\beta$. It is elementary to show that $M_\beta(X)$ is in fact equal to $M_\beta$. For example, the proof of \cite[Lemma 4.24]{sahin} works under our assumption that $F$ is a stable sheaf on $X=\tot (K_S)$ when $-K_S$ is ample. It is well-known that $M_\beta(X)$ is equipped with a symmetric obstruction theory and hence a virtual cycle $ [M_\beta(X)]^{\rm vir}\in A_0(M_\beta(X))$. Toda in \cite{Toda} proved that $n_\beta=\deg [M_\beta(X)]^{\rm vir}$ whenever the GW/PT correspondence holds, which is the case for local del Pezzo surfaces.
Since $M_\beta$ is smooth of dimension $\beta^2+1$, we have  $\deg [M_\beta(X)]^{\rm vir}= (-1)^{\beta^2+1} e(M_\beta)$. 
\end{proof}

\begin{remark}
In \S~\ref{sec:refined}, we discuss an $sl_2\times sl_2$-action on the cohomology of $M_\beta$ which will allow us to refine the $n^g_\beta$.
\end{remark}

\subsection{Blowup property of the BPS invariant} \label{sec:blowup}
Let $\pi :S_{r+1}\to S_r$ be the blowup at a general point $p\in S_r$. Let $\beta$ be a divisor on $S_{r}$. In this section, we show that $M_\beta(S_r)$ and $M_{\pi^*\beta}(S_{r+1})$ are isomorphic. Consequently, $n_{\pi^*\beta}(S_{r+1})=n_\beta(S_r)$. 

Let $F \in M_\beta(S_r)$. We let $\mathrm{supp}(F)\in|\beta|$ be the support scheme defined by
the Fitting ideal.  Since $F$ has pure dimension~1 we have a
presentation of $F$ as
\begin{equation}
0\to E_1\stackrel{\phi}{\to} E_2\to F\to 0,
\label{eq:fitting}
\end{equation}
where $E_1$ and $E_2$ are locally free of the same rank.  Then
$\mathrm{supp}(F)\in|\beta|$ is the subscheme of $S$ defined by the vanishing
of $\det\phi$, and is well known to be independent of the choice of resolution.  This
defines the Chow morphism
\[
M_\beta\to |\beta|.
\]

\begin{lemma} \label{lem:pistarpure} For a pure one-dimensional sheaf $F$ on $S_r$, $\pi^*F$ is pure and
  $$\mathrm{supp}(\pi^*F)=\pi^*\mathrm{supp}(F).$$
\end{lemma}
\begin{proof}
Choose a presentation (\ref{eq:fitting}) of $F$.
Applying $\pi^*$ we
get
\begin{equation}
0\to \pi^*E_1\stackrel{\pi^*(\phi)}{\longrightarrow} \pi^*E_2\to \pi^*F\to 0.
\label{eq:pifitting}
\end{equation}
The sequence \eqref{eq:pifitting} is exact on the left because  the kernel of $\pi^*(\phi)$ is zero away from the exceptional curve as $\pi^*(\phi)$ can be identified with $\phi$, which implies that the kernel is zero everywhere since $\pi^*E_1$ is locally free.

Then $\pi^*F$ is pure by the Auslander-Buchsbaum formula, and
$\mathrm{supp}(\pi^*F)$ is the subscheme of ${S}_{r+1}$ defined by the
vanishing of $\det\pi^*(\phi)=\pi^*\det\phi$, which is equal to the scheme-theoretic
inverse image $\pi^*(\mathrm{supp}(F))$. 
\end{proof}
Thus, we can think of $F$ and $\pi^*F$ having ``the same'' support,
after identifying $|\beta|$ with $|\pi^*\beta|$ via $C\mapsto\pi^*C$.

\begin{lemma}
 For a pure one-dimensional sheaf $F$ on $S_r$,  $\pi_*\pi^*F\simeq F$, $R^1\pi_*\pi^*F=0$, and $\chi(\pi^*F)=\chi(F)$.
\label{lem:pistar}
\end{lemma}
\begin{proof}
We apply $\pi_*$ to (\ref{eq:pifitting}), noting by the projection formula
that $\pi_*\pi^*E_i\simeq E_i$ and $R^j\pi_*\pi^*E_i=0$ for $i=1,2$ and $j>0$, 
since each $E_i$
is locally free and $R^j\pi_*\cO_{{S}_{r+1}}=0$.  We obtain
\begin{equation}
0\to E_1\stackrel{\phi}{\to}E_2\to \pi_*\pi^*F\to0 
\label{eq:pushpull}
\end{equation}
and $R^1\pi_*\pi^*F=0$.  Comparing (\ref{eq:pushpull}) and (\ref{eq:fitting}) we see
that $\pi_*\pi^*F\simeq F$.  Finally, by Leray
\[
\chi(\pi^*F)=\chi(\pi_*\pi^*F)-\chi(R^1\pi_*\pi^*F)=\chi(F).
\]
\end{proof}

\begin{proposition}\label{prop:blpback}
 For a pure one-dimensional sheaf $F$ on $S_r$ with $\chi(F)=1$, $F$ is stable if and only if $\pi^*F$ is stable.
\label{prop:stable}
\end{proposition}
\begin{proof}
Suppose $\pi^*F$ is stable. Let $G$ be a saturated subsheaf of $F$. Since $F/G$ is pure, by \eqref{eq:pifitting}, $L_1\pi^*(F/G)=0$ and hence $\pi^*G$ is a subsheaf of $\pi^*F$. By Lemma~\ref{lem:pistar}, if $G$ destabilize $F$, then $\pi^*G$ destabilize $\pi^*F$. Hence $F$ is also stable. 

Conversely, suppose that $F$ is stable and $G\subset\pi^*F$ is a subsheaf with $\chi(G)\ge 1$.
Applying $\pi_*$ we get that $\pi_*G$ is a subsheaf of  $\pi_*\pi^*F
\simeq F$.  Since $R^1\pi_*G$ is supported at $p$, we have 
$\chi(R^1\pi_*G)=h^0(R^1\pi_*G)\ge0$ and
\[
\chi(\pi_*G)=\chi(G)+\chi(R^1\pi_*G)\ge\chi(G)\ge1
\]
by Leray.  Hence $\pi_*G$ destabilizes $F$, contradicting the stability of $F$.
\end{proof}

\begin{lemma}
  Let $F\in M_{\pi^*\beta}(S_{r+1})$.  Then $\pi_*F$ is pure with $c_1(\pi_*F)=\beta$.
  \label{lem:pushpure}
\end{lemma}
\begin{proof}
Clearly $\pi_*F$ can only have torsion at $p$.  Suppose we had  a skyscraper
sheaf $\CC_p\subset \pi_*F$ at $p$.  This gives a global section $s$ of $\pi_*F$ annihilated
by $m_p$.  Then $s$ corresponds to a global section $\tilde{s}$ of $F$, which is then
necessarily annihilated by
$\pi^{-1}(m_p)\cO_{S_{r+1}}=\shI_E$.  So $\tilde{s}$ induces a map $\cO_{ S_{r+1}}/\shI_E\simeq
\cO_E\rightarrow F$ which is injective because $\cO_E$ is pure.  Since $\chi(\cO_E)=1$, this would violate stability of $F$.

The class of $\pi_*F$ is $\beta$ since  $\pi_*F|_{S_r-p}$ is identified with $F|_{{S}_{r+1}-E}$
via $\pi$ and the restriction map 
\[ \mathrm{Pic}(S_{r+1})\to \mathrm{Pic}(S_{r+1} -E) \simeq \mathrm{Pic}(S_r-p)\simeq \mathrm{Pic}(S_r)\] 
is the left inverse of $\pi^*:\mathrm{Pic}(S_r)\to  \mathrm{Pic}(S_{r+1})$.
\end{proof}

Now we consider the natural map $\rho:\pi^*\pi_* F\to F$.  

\begin{proposition}  
\label{prop:pushforward}
Let $F\in M_{\pi^*\beta}(S_{r+1})$. Then $\rho:\pi^*\pi_* F\to F$ is an isomorphism. Consequently, $\pi_*F$ is stable and $\chi(\pi_*F)=1$.  
\end{proposition}
\begin{proof}  
Let $G$ be the kernel of $\rho$.  Then $\pi_*G$ is a subsheaf of
$\pi_*\pi^*\pi_*F$.   But $\pi_*\pi^*\pi_*F\simeq\pi_*F$ by the first statement
of Lemma~\ref{lem:pistar} applied to $\pi_*F$. 
But $\pi_*G$ is supported at $p$, contradicting the purity of $\pi_*F$ unless 
$\pi_*G=0$.

By Lemma~\ref{lem:pistarpure} and Lemma~\ref{lem:pushpure}
we see that $c_1(\pi^*\pi_*F)=\pi^*(\beta)$, so 
letting 
$Q=\mathrm{coker}(\rho)$, it follows that 
$c_1(G)=c_1(Q)$. 

Finally, we compute the Euler characteristics of $\pi_*F$ in two different ways.  Using $\rho$, 
we see that 
$\chi(\pi^*\pi_*F)=1+\chi(G)-\chi(Q)$.  
By Lemma~\ref{lem:pistar}, we see that
$R^1\pi_*(\pi^*\pi_*F)=0$.  So by Leray for $\pi_*$, we compute
$\chi(\pi^*\pi_*F)=\chi(\pi_*F)$.  But again by Leray, we have  $\chi(\pi_*F)=
\chi(F)+h^0(R^1\pi_*F)=1+h^0(R^1\pi_*F)$.
From $\pi_*G=0$ it follows that $\chi(G)\leq 0$, 
hence $\chi(Q)\leq 0$. 
We must have $Q=0$ since $F$ is stable, 
and $G$ is $0$-dimensional by $c_1(G)=c_1(Q)$. 
Again by $\pi_*G=0$ we have $G=0$, 
so $\rho$ is an isomorphism. 
\end{proof}


\begin{proposition}\label{prop:pullback}
Let $\pi :S_{r+1}\to S_r$ be a blowup. Let $\beta$ be a divisor on $S_{r}$. Then, $M_\beta(S_r)$ and $M_{\pi^*\beta}(S_{r+1})$ are isomorphic. 
\end{proposition}
\begin{proof}
Let $\shF$ be a universal family on $M_\beta(S_r)\times S_r$. The pullback $\shF' = (id\times \pi)^* \shF$ is a family on $M_\beta(S_r)\times S_{r+1}$, whose fibers are stable sheaves in $M_{\pi^*\beta}(S_{r+1})$ by Proposition \ref{prop:blpback}. So $\shF'$ induces the morphism $\pi^*:M_\beta(S_r)\to M_{\pi^*\beta}(S_{r+1})$. By Proposition \ref{prop:pushforward}, $\pi^*$ is bijective and since $\pi$ is an isomorphism away from the exceptional divisor, $\pi^*$ is a birational morphism. By Proposition \ref{prop:bps equiv}, the two moduli spaces $M_\beta(S_r)$ and $M_{\pi^*\beta}(S_{r+1})$ are smooth. Therefore by Zariski's main theorem, $\pi^*$ is an isomorphism. 
\end{proof}


In \S~\ref{sec:localcal}, we will consider curve classes $\beta$ of arithmetic genus at most 2. By Lemma~\ref{lem:Rocco} and Proposition~\ref{prop:pullback}, with a few exceptions it is enough to calculate BPS numbers for very ample classes $\beta$ by blowing down all $(-1)$-curves $l$ such that $\beta.l=0$. 

\begin{remark}
The isomorphism constructed above commutes with the Chow morphisms. Therefore the higher genus BPS invariants as well as their $sl_2\times sl_2$ refinements as defined in \cite{KL,Joyce,MT} remain unchanged as well. 
\end{remark}

\subsection{$\delta$-stable pairs and wall-crossing}\label{sec:pairs}

Suppose that the BPS invariants $n^g_\beta(X)=n^g_\beta$ satisfy the Gopakumar-Vafa formula
\begin{equation}\label{eq:gwbps}
\sum_{\beta, g} \mc{I}_{\beta}^g(X) q^\beta \lambda^{2g-2}
=\sum_{\beta, g, k}n^g_\beta \frac{1}{k}\left(2
\sin\left(\frac{k\lambda}{2}\right)^{2g-2} q^{k\beta} \right),
\end{equation}
where $\mc{I}_{\beta}^g(X)$ are the (local) Gromov-Witten invariants.
Using the conjectured GW-PT correspondence,\footnote{The GW-PT correspondence is proven when $S$ (and hence $X$) is toric by combining the toric GW-DT correspondence \cite{MOOP} with the DT-PT correspondence \cite{Bridgeland}. The GW-PT correspondence for a general del Pezzo surface $S$ reduces to the toric cases by taking a toric blowup of $\ptwo$ and then using deformation invariance of the GW and PT invariants.} we would then have the following PT-BPS formula \cite{katz}

\begin{equation}\label{productGV}
\begin{split}
Z_{PT}=\prod_\beta &\left(  \prod_{j=1}^\infty   \left(1+(-1)^{j+1}q^j Q^\beta\right)^{jn^0_\beta}\right. \\
& \left. \cdot\prod_{g=1}^\infty\prod_{k=0}^{2g-2}\left(1+\left(-1\right)^{g-k}q^{g-1-k}Q^\beta \right)^{\left(-1\right)^{k+g}n^g_\beta {2g-2 \choose k}}\right), 
\end{split}
\end{equation}
where $Z_{PT}$ is the generating function for the PT invariants. See Definition~\ref{def:pt} for PT-stable pairs.

In \cite{kkv}, Katz, Klemm and Vafa developed a geometric computational technique for BPS invariants. Later in \cite{ckk}, the refined BPS invariants are defined from the refined PT invariants and the method is extended to compute the refined BPS indices.

As a consequence of the product formula \eqref{productGV}, it was suggested in \cite{kkv} that the genus zero BPS invariant $n_\beta:=n_\beta^0$ can be computed by
\begin{equation}\label{eq:kkv} n_\beta= PT_{\beta,1}-PT_{\beta,-1} + \text{correction terms},\end{equation}
where $PT_{\beta,n}$ is the stable pair invariant of degree $\beta$ and Euler characteristic $n$. The correction terms are combinations of lower degree PT invariants. In \cite{ckk}, the correction terms are interpreted as a wall-crossing contribution of the moduli spaces of stable pairs. After wall-crossing, the moduli spaces of stable pairs are related to the moduli space of (Gieseker-)stable sheaves.  We will formulate and prove a refined version of (\ref{eq:kkv}) in Proposition~\ref{prop:eulermd} below.

To compute the local BPS invariants, we will use Proposition~\ref{prop:bps equiv} and compute the topological Euler characteristic of the moduli spaces $M_\beta$. 
More generally, we compute the Poincar\'{e} polynomials.
\begin{definition}
For a complex algebraic variety $M$, we let $E_M(u,v)$ be its E-polynomial.  We define the \emph{virtual Poincar\'{e} polynomial} of $M$ as the polynomial $P_t(M)=E_M(-t^{1/2},-t^{1/2})$ in $t^{1/2}$.  The virtual Poincar\'e polynomial satisfies the properties
\begin{itemize}
\item[(i)] $ P_t(M)= \sum_{i\ge 0} \dim_{\mathbb{Q}} H^i(M, \mathbb{Q})\,t^{i/2}$ if $M$ is nonsingular and projective. 
\item[(ii)] $P_t(M) = P_t(M\setminus Y) +P_t(Y)$ for a closed algebraic subset $Y$ of $M$. 
\end{itemize}
\end{definition} 
In our cases, the odd cohomology groups of $M_\beta$ vanish so that $P_t(M_\beta)$ is actually a polynomial in $t$ and the topological Euler characteristic is given by $\euler(M)=P_1(M)$. Note that although the Poincar\'{e} polynomial is not motivic in general, the virtual Poincar\'{e} polynomial is motivic, and since $M_\beta$ is smooth, the virtual Poincar\'{e} polynomial agrees with the usual Poincar\'{e} polynomial in $t$ with $t^{1/2}$ substituted for $t$. 

We will freely use the following properties of the virtual Poincar\'{e} polynomial, which follow from the definition (See \cite[\S 4.5]{Fulton}).
\begin{itemize}
\item[(iii)] If $M$ is a disjoint union of a finite number of locally closed subvarieties $M_i$, then $P_t(M)=\sum P_t(M_i)$.
\item[(iv)] If $M$ is a Zariski locally trivial fibration over $Y$ with fibers $F$, then $P_t(M)=P_t(F)P_t(Y)$.
\end{itemize}

To compute the Poincar\'{e} polynomial, we relate $M_\beta$ birationally with the moduli spaces of $\delta$-stable pairs by wall-crossing. This approach is taken in \cite{cc} to compute the Betti numbers for $M_\beta$ when $S=\PP^2$ and $\beta=4$ and $5$. See also \cite{ckk}.

\begin{definition}\label{def:pt0}
A \emph{pair} on $X$ is a pair $(s,F)$ of a coherent sheaf $F$ on $X$ of class $\beta$ together with a nonzero section $s\in H^0(F)$. A morphism between pairs is a morphism of sheaves which preserves the sections up to multiplication by a constant. 
\end{definition}

The topological data of $(s,F)$ are defined to be those of the sheaf $F$. The notion of pairs originated in the work of Le Potier \cite{LeP} on \emph{coherent systems}. A coherent system is a pair $(V,F)$ of a coherent sheaf $F$ with a subspace $V\subset H^0(F)$ of fixed dimension. So, our pairs are coherent systems of dimension one. It is often convenient to consider a sheaf as a coherent system of dimension zero. 

Le Potier \cite{LeP} studied a series of stability conditions on coherent systems, which reads as follows for pairs. See also \cite{mhe}.

\begin{definition}
Let $\delta\in \mathbb{Q}_+$. A pair $(s,F)$ is \emph{$\delta$-stable} if 
\begin{itemize}
\item[(a)] $F$ is pure. 
\item[(b)] For any proper nonzero subsheaf $G$ of $F$, we have \[ \frac{\chi(G)+\epsilon(s,G)\delta}{r(G)} <\frac{\chi(F)+\delta}{r(F)},\]
where $r(F)$ is the linear coefficient of the Hilbert polynomial of $F$ and $\epsilon(s,G)=1$ if $s$ factors through $G$ and $\epsilon(s,G)=0$ otherwise. 
\end{itemize}
When the equality is allowed in Condition (2), then the pair is \emph{$\delta$-semistable}. 
\end{definition}

As in \S\ref{sec:bps}, we use the ample line bundle $L=-K_S$ to define the Hilbert polynomial of a sheaf on $S$. So, $r(F)=(-K_S).[F]$. We denote by $M^\delta_{\beta, n}$  the moduli space of $\delta$-stable pairs $(s,F)$ on $S$ with $[F]=\beta$ and $\chi(F)=n$. When there are no strictly semistable $\delta$-stable pairs, $M^\delta_{\beta, n}$ is constructed as a projective scheme by GIT. 

The values of $\delta$ where there exist strictly $\delta$-semistable pairs are called the \emph{walls}. Then the moduli space $M^\delta_{\beta, n}$ changes only at walls. We will see that in our cases there are only finitely many walls.  

One special case is when $\delta$ is sufficiently large, which we denote by $\delta=\infty$. In this case, the $\delta$-stability condition is equivalent to the stability condition on pairs of Pandharipande and Thomas \cite{pt}.

\begin{definition}\label{def:pt}
A pair $(s,F)$ is PT-stable if \begin{itemize}
\item[(a)] $F$ is pure of dimension 1. 
\item[(b)] The cokernel of $s:\shO_X \to F$ is zero-dimensional. 
\end{itemize}
\end{definition}

We denote by $P_n(S,\beta)$ the moduli space of PT-stable pairs on $S$.  In other words, $P_n(S,\beta)=M^\infty_{\beta, n}$. By condition (2) in Definition~\ref{def:pt}, it is straightforward to see that $P_n(S,\beta)$ is empty when $n< 1-p_a(\beta)$.
Pandharipande and Thomas \cite{pt} proved that  $P_n(X,\beta)$ is equipped with a symmetric obstruction theory when $X$ is a Calabi-Yau threefold. In general, when $X=\tot (K_S)$, $P_n(X,\beta)$ may not be equal to $P_n(S,\beta)$. However, we will only consider the wall-crossing of $\delta$-stable pairs defined on $S$.

\begin{proposition}\label{prop:PT} 
Let $S$ be a del Pezzo surface. Let $p_a=p_a(\beta)$ and assume $n\ge 1-p_a$. Recall that $w=(-K_S).\beta$.  
\begin{itemize}
\item[(i)] If $\beta$ is a line class, then $P_n(S,\beta)\simeq \PP^{n-1}$.
\item[(ii)] Assume that $\beta$ is base point free. If $\beta$ is $(n-2+p_a)$-very ample, then $P_n(S,\beta)$ is a projective bundle of rank $w-n$ over the Hilbert scheme $\mathrm{Hilb}^{n-1+p_a}(S)$.
\end{itemize}
\end{proposition}

\begin{proof}
The proof is essentially same as that of \cite[Lemma 2.3]{cc}. By \cite[Proposition B.8]{ptbps}, $P_n(S,\beta)$ is isomorphic to the space of pairs $(C,Z)$ where $C$ is a curve in class $\beta$ and $Z$ is a subscheme of $C$ of length $n-1+p_a$. In particular, the assertion for a line class $\beta$ is straightforward. Note that each line class contains a unique line. 

Now, let $\shI$ be the universal ideal sheaf on $\mathrm{Hilb}^{n-1+p_a}(S)\times S$ and let $p:\mathrm{Hilb}^{n-1+p_a}(S)\times S \to\mathrm{Hilb}^{n-1+p_a}(S) $ and $q:\mathrm{Hilb}^{n-1+p_a}(S)\times S \to S$ be the projections. Then $P_n(S,\beta)$ is the projective bundle $\PP(p_*(\shI \otimes q^*\shO_S(\beta)))$ provided that $p_*(\shI \otimes q^*\shO_S(\beta))$ is locally free. Since $\beta$ is $(n-2+p_a)$-very ample, we have $H^1(I_Z\otimes \shO(\beta))\simeq H^1(\shO(\beta))$ for any subscheme $Z$ of length $n-1+p_a$. The latter space vanishes by Lemma~\ref{lem:nefbig}. 
By the semicontinuity theorem, $p_*(\shI \otimes q^*\shO_S(\beta))$ is locally free and hence $P_n(S,\beta)$ is a projective bundle.

Since $H^1(I_Z\otimes \shO(\beta))\simeq H^1(\shO(\beta))=0$ for base point free $\beta$, the rank of the projective bundle $\PP(p_*(\shI \otimes q^*\shO_S(\beta)))$ can be computed by Riemann-Roch. 
\end{proof}

On the other extreme when $\delta$ is sufficiently small, which we denote by $\delta=0^+$, it is elementary to check that for $(s,F)\in M^{0^+}_{\beta, n}$, the sheaf $F$ is a stable sheaf provided that $(-K_S).\beta$ and $n$ are coprime. In this case, we have a forgetful map \[\xi:M^{0^+}_{\beta, n} \to M_{\beta,n}.\] In what follows, we only consider the case where $n$ is either $1$ or $-1$, so the coprime condition is always satisfied. 

\begin{proposition}\label{prop:eulermd}
$ P_t(M_\beta) =P_t(M^{0^+}_{\beta, 1}) -  t P_t(M^{0^+}_{\beta, -1})$.
\end{proposition}

\begin{proof}
This formula is proven for $S=\IP^2$ in \cite[Lemma 5.1]{cc}. The same proof applies to general del Pezzo surfaces. We sketch the proof here. 

Let $n$ be either $1$ or $-1$. Let $(M_{\beta,n})_k$ (resp. $(M^{0^+}_{\beta, n})_k$) denote the locus in $M_{\beta,n}$ (resp. $M^{0^+}_{\beta, n}$) defined by the condition $h^0(F)=k$. Then the forgetful map $\xi$ restricted to $(M^{0^+}_{\beta, n})_k$ is a Zariski locally trivial $\PP^{k-1}$-fibration since any nonzero section of $F$ defines a $0^+$-stable pair and an automorphism of a stable sheaf is given by scalar multiplication. Therefore we have 
\[P_t(M^{0^+}_{\beta, n}) =\sum_k P_t((M^{0^+}_{\beta,n})_k)= \sum_k P_t(\PP^k) P_t((M_{\beta,n})_k).\] 

For a sheaf $F\in M_\beta$, we define its dual by $F^D=\mathcal{E}xt^1(F, \omega_S)$. Since $F$ is a pure one-dimensional sheaf, $F^{DD}\simeq F$ (\cite[Proposition 1.1.10]{HL}). The local-to-global spectral sequence $E^{pq}_2=H^p(\mathcal{E}xt^q(F, \omega_S))$ degenerates at level two and hence $h^i(F^D)=h^{1-i}(F)$ for $i=0,1$. Thus the association $F\mapsto F^D$ induces an isomorphism between $(M_{\beta, 1})_k$ and $(M_{\beta, -1})_{k-1}$. The fact that this association is a morphism of schemes is proved in \cite{Maican} when $S=\PP^2$, but the same proof applies to a general del Pezzo surface $S$. 

Therefore we have
\begin{align*}
P_t(M^{0^+}_{\beta, 1}) -  tP_t(M^{0^+}_{\beta, -1}) &= \sum_k P_t(\PP^{k-1}) P_t((M_{\beta,1})_k)- t P_t(\PP^{k-1}) P_t((M_{\beta,-1})_k) \\
&= \sum_k P_t(\PP^{k-1}) P_t((M_{\beta,1})_k)- tP_t(\PP^{k-1}) P_t((M_{\beta,1})_{k+1})\\
&= \sum_k (P_t(\PP^{k-1})-tP_t(\PP^{k-2})) P_t((M_{\beta,1})_k)\\
&= \sum_k P_t((M_{\beta,1})_k) = P_t(M_\beta).
\end{align*}
\end{proof}

Proposition~\ref{prop:eulermd} suggests that the correction terms in \eqref{eq:kkv} come from wall-crossing on $\delta$-stable pairs. More detail on the correspondence between wall-crossing terms and the correction terms can be found in \cite[\S 9.3]{ckk}.

Now we study how the moduli space changes when we cross a wall. Let $\delta_0$ be a wall and let $\delta_-$ and $\delta_+$ be rational numbers sufficiently close to $\delta_0$ such that $\delta_-<\delta_0<\delta_+$ and there are no walls between $\delta_-$ and $\delta_+$ other than $\delta_0$. We want to compare $M^{\delta_+}_{\beta, n}$ and $M^{\delta_-}_{\beta, n}$. 

Let $(s,F)$ be a $\delta_+$-stable pair which is not $\delta_-$-stable. Let $F''$ be a subsheaf of $F$ such that $s$ factors through $F''$. So $\frac{\chi(F'')+\delta_+}{r(F'')} < \frac{\chi(F)+\delta_+}{r(F)}$. Since $r(F'')\le r(F)$, this implies $\frac{\chi(F'')+\delta_-}{r(F'')} < \frac{\chi(F)+\delta_-}{r(F)}$. Therefore for $(s,F)$ to be not $\delta_-$-stable, there must be a subsheaf $F''$ of $F$ such that \[\frac{\chi(F'')}{r(F'')} > \frac{\chi(F)+\delta_-}{r(F)}.\]
Necessarily, the section $s$ does not factor through $F''$. Thus we have an exact sequence of pairs
\begin{equation}
\label{eq:delta+} 0\to (0,F'') \to (s,F) \to (s',F') \to 0, 
\end{equation}
where $F'=F/F''$ and $s'$ is the section on $F'$ induced by $s$. Here, $(0,F'')$ denotes the sheaf $F''$ considered as a coherent system of dimension zero. 

On the other hand, if $(\tilde s,\tilde F)$ is a $\delta_-$-stable pair which is not $\delta_+$-stable, by the same reasoning, we have an exact sequence 
\begin{equation}
\label{eq:delta-}0\to (s',F') \to (\tilde s,\tilde F) \to (0,F'')\to 0. 
\end{equation}

The wall $\delta_0$ is called a \emph{simple wall} if $(s',F')$ is $\delta_0$-stable and $F''$ is stable (as a sheaf) so that there are no further decompositions to be considered. 
In this paper, we will only consider the cases where all walls are simple walls. 

To denote a decomposition as in \eqref{eq:delta+} and \eqref{eq:delta-}, we use the notation
\begin{equation}\label{eq:decomp}
(1,(\beta,n))=(1,(\beta',n'))+(0,(\beta'',n'')),
\end{equation}
where $\beta'=[F']$, $\beta''=[F'']$, $n'= \chi(F')$ and $n''=\chi(F'')$. 
So if there is a wall-crossing for $M^{\delta}_{\beta, n}$, we must have a decomposition \eqref{eq:decomp} such that $M_{\beta'', n''}$ and $M^{\delta_0}_{\beta', n'}$ are nonempty, where $$ \frac{n+\delta_0}{(-K_S).\beta}=\frac{n'+\delta_0}{(-K_S).\beta'}=\frac{n''}{(-K_S).\beta''}.$$ In such a case, the pairs in $M^{\delta_+}_{\beta, n}$ of the form \eqref{eq:delta+} parametrized by $\IP(\Ext^1( (s',F') ,(0,F'')))$ are replaced with the pairs in $M^{\delta_-}_{\beta, n}$ of the form \eqref{eq:delta-} parametrized by $\IP(\Ext^1( (0,F''),(s',F')))$. This wall-crossing phenomenon can be explained by \emph{elementary modification} of pairs. See \cite[\S 3]{thad}, \cite[Lemma.4.24]{mhe} and \cite{cc}. Now each Ext group can be computed using the following proposition. 

\begin{proposition}\cite[Corollary 1.6]{mhe}\label{prop:sespair}
Let $\Lambda=(s, F)$ and $\Lambda'=(s', F')$ be pairs on $X$. Then there is a long exact sequence
\begin{align*}
0&\to \Hom (\Lambda,\Lambda')\to \Hom (F,F')\to H^0(F')/\langle s'\rangle\\
&\to \Ext^1(\Lambda,\Lambda')\to \Ext^1(F,F')\to H^1(F')\\
&\to \Ext^2(\Lambda,\Lambda')\to \Ext^2(F,F')\to H^2(F')\to \cdots.
\end{align*}
\end{proposition}

\section{Calculations of local BPS numbers}\label{sec:localcal}

We calculate the local BPS numbers by applying the wall-crossing techniques described in the previous sections. In this section, we assume that $\beta$ is either a line class, a conic class or a nef and big curve class so that there is a smooth irreducible curves in class $\beta$. 

When $\beta$ is nef and big, we have $H^i(\beta +K_S)=0$ for $i>0$, which is due to Ramanujam \cite{Ram}, \cite[Theorem 4.3.1]{Laz}. Therefore 
\begin{equation}\label{eq:h0}
h^0(\beta+K_S)=\chi(\beta+K_S)=\frac12 (\beta+K_S)\beta+1 = p_a(\beta).
\end{equation}

\subsection{Arithmetic genus 0}\label{sec:localg0}
For a nef and big curve class of arithmetic genus $0$ on $S=S_r$, $\beta+K_S$ is not nef 
since $(\beta+K_S)\beta=-2$. 
Hence if $r \ge 2$, there is a line $l$ on $S$ with $(\beta+K_S).l<0$,
and $\beta.l=0$ follows from the nefness of $\beta$. 
By blowing down such lines, we see that $\beta$ is a pullback of the class $(1)$ or $(2)$ on $\PP^2$, the class $(d; d-1)$ on $S_1$ with $d\ge 2$,
or the class $(1, k)$ on $\pone\times\pone$ with $k\ge1$.

\begin{proposition}\label{prop:g0}
  Let $\beta$ be a curve class on $S$ of arithmetic genus 0. If $\beta$ is either a line class, a conic class or a nef and big curve class, then $M_\beta$ is isomorphic to $\PP^{w-1}$. 
\end{proposition}
\begin{proof}

In the  nef and big case, 
we may assume that $(S, \beta)$ is $(\PP^2,(1))$, $(\PP^2,(2))$, $(S_1, (d; d-1))$ with $d\ge2$, 
or $(\pone\times\pone, (1, k))$ with $k\ge1$ 
by Proposition~\ref{prop:pullback} and the preceding discussion. 

Let $F$ be a stable sheaf with $\chi(F)=1$. Then there is a nonzero section $s \in H^0(F)$ which induces a morphism $i:\shO_{S} \to F$. Let $C'$ be the curve on $S$ defined by the kernel of $i$. Put $\beta' = [C']$. Then if $\beta'\ne \beta$, stability is contradicted because $p_a(\beta')\le 0$
as can be seen using the description of $\beta$ in each case. 
We conclude that $\beta' = \beta$ and $F \simeq \shO_C$ where $C$ is in class $\beta$. Therefore $M_\beta$ is isomorphic to the complete linear system $|\shO(\beta)|\simeq \PP^{w-1}$.
\end{proof}

\begin{Corollary}
$P_t(M_\beta)=\displaystyle\frac{1-t^w}{1-t}$ and $n_\beta=(-1)^{w-1}w$.
\end{Corollary}

\subsection{Arithmetic genus 1}\label{sec:localg1}

By Proposition~\ref{prop:pullback}, blowing down all lines $l$ with $\beta.l=0$ does not change the moduli space of stable sheaves. 

\begin{lemma}\label{lem:g1classes}
Let $\beta$ be a nef and big curve class on a del Pezzo surface $S$ of arithmetic genus $1$ such that $\beta.l\ge 1$ for all line classes $l$. Then $\beta=-K_{S_r}$ for $0\le r \le 8$ or $\beta=-K_{\pone\times\pone}$.
\end{lemma}
\begin{proof}
We have $\beta(\beta+K_S)=2p_a(\beta)-2=0$. By \eqref{eq:h0}, we have $h^0(\beta+K_S)=1$. Therefore $\beta+K_S$ is effective. 
Hence it is enough to show that $\beta$ is ample. 

If $S=S_r$ with $r\geq 2$, 
$\beta$ is ample  
from the assumption that $\beta.l\ge 1$ holds for all line classes $l$. 
On $\ptwo$ or $\pone\times\pone$, 
any nef and big class is ample.  
On $S_1$, 
$\beta=(d; a)$ satisfies $\beta.(h-e_1)=d-a\geq 1$.
It follows that $d>a$, and $\beta$ is ample. 
\end{proof}

Suppose two distinct lines $l_1$ and $l_2$ satisfy $\beta.l_1=\beta.l_2=0$, then since $\beta(l_1+l_2)=0$, by the Hodge index theorem, $(l_1+l_2)^2<0$, which implies $l_1.l_2=0$. Therefore they are mutually disjoint and the number of them is at most $r$.

After blowing down all lines $l$ with $\beta.l=0$, we may assume that $\beta=-K_S$. When $0\le r\le 7$, $\beta=-K_S$ is base point free. The case $\beta=-K_{S_8}$ is the only case where $\beta$ is neither base point free nor can be blown down to a base point free curve class. We will study this exceptional case in Example~\ref{exam:KS8}.

\begin{proposition} \label{local bps g1}
Let $\beta$ be a nef and big curve class on $S=S_r$ of arithmetic genus 1 and $\beta\ne -K_{S_8}$. Let $\eta$ be the maximum number of disjoint lines $l$ such that $\beta. l=0$. Then $$P_t(M_\beta)=\displaystyle\frac{1-t^w}{1-t}\left(1+(\euler(S)-2-\eta)t +t^2\right)$$ and $n_\beta= (-1)^{w-1} w (\euler(S)-\eta)$.
\end{proposition}
\begin{proof}
Let $\pi:S\to S'$ be the blowing-down of all lines $l$ such that $\beta.l=0$. By Proposition~\ref{prop:pullback}, $M_\beta(S)\simeq M_{\pi_*\beta}(S')$. By the remark before Proposition~\ref{prop:PT}, $P_{-1}(S', \pi_*\beta)$ is empty, which implies that the forgetful map $\xi:M^{0^+}_{\pi_*\beta, 1}(S') \to M_{\pi_*\beta}(S')$ is an isomorphism by the proof of Proposition \ref{prop:eulermd}. Hence, $M_\beta(S)$ is isomorphic to $M^{0^+}_{\pi_*\beta, 1}(S')$.

If $\beta\ne -K_{S_8}$, then by Lemma \ref{lem:g1classes} $S'$ is either $\pone\times\pone$ or $S_r$ with $0\le r\le 7$ and $\pi_*\beta=-K_{S'}$, which is base point free. So, by Proposition~\ref{prop:PT}, $P_1(S',\pi_*\beta)$ is a $\PP^{w-1}$-bundle over $S'$. One can check that there is no wall-crossing for stable pairs in this case so that $M^{0^+}_{\pi_*\beta, 1}(S')$ is isomorphic to $M^{\infty}_{\pi_*\beta, 1}(S')= P_1(S',\pi_*\beta)$. Indeed, at a wall $\delta_0$, we have a decomposition of the form
\begin{equation}\label{eq:decomp2}
(1,(\beta,1))=(1,(\beta',n'))+(0,(\beta'',n'')),
\end{equation}
where $\beta=\beta'+\beta''$, $n'+n''=1$ and $\delta_0=\frac{w}{(-K)\beta''} n''-1$.  Since $\delta_0$ must be positive, we see that $n'=0$ and $n''=1$. Now to have a nontrivial wall-crossing, there must be a sheaf $F'$ with $[F']=\beta'$ and $\chi(F')=0$, which in addition has a nontrivial section. Consequently, $p_a(\beta')\ge 1$. But one can numerically check that when $\beta=-K_{S'}$, such a decomposition does not exist. See for example the list of curve classes in \cite[Appendix A]{kkv}.

Therefore $M_\beta(S)$ is isomorphic to a $\PP^{w-1}$-bundle over $S'$, hence the results follow. 
\end{proof}

\begin{remark}
Blowing down in the proof of Proposition~\ref{local bps g1} corresponds to the wall-crossing in pairs. When $\beta$ is nef and big, by Proposition~\ref{prop:PT}, $P_1(S, \beta)$ is a $\PP^{w-1}$-bundle over $S$. 
For each line $l$ such that $\beta.l=0$, we have a decomposition of the form 
\begin{equation}\label{eq:g1decomp}(1,(\beta,1))=(1,(\beta',0)) +(0,(l,1)),\end{equation}
where $\beta'$ is a curve of arithmetic genus 1 and $\beta'.l=1$. The corresponding wall is at $\delta_0=w -1>0$. 
A pair in $(1,(\beta',0))$ is of the form $(s, \shO_{\beta'})$ and a pair in $(0,(l,1))$ is of the form $(0, \shO_l)$. 

By using Proposition~\ref{prop:sespair}, we have 
\begin{align*}
\Ext^1(  (s, \shO_{\beta'}), (0, \shO_l)   )&\simeq  \CC^2,\\
\Ext^1( (0, \shO_l) , (s, \shO_{\beta'})  )&\simeq  \CC.
\end{align*}
Hence by wall-crossing at $\delta_0$, $\eta$ copies of a $\PP^1$ bundle over $\IP^{w-1}\times \IP^0$ in $M^{\delta_+}_{\beta, 1}$ are replaced with $\eta$ copies of a $\PP^0$ bundle over $\IP^{w-1}\times \IP^0$ in $M^{\delta_-}_{\beta, 1}$. One can check that this wall-crossing is in fact a blow-up $\rho: M^{\delta_+}_{\beta, 1} \to M^{\delta_-}_{\beta, 1}$ along the locus isomorphic to $\eta$ copies of $\IP^{w-1}\times \IP^0$. 

Therefore, 
$$P_t(M_\beta)=\displaystyle\frac{1-t^w}{1-t}\left((1+(e(S)-2)t +t^2) - \eta (1+t-1)\right)$$
as required.
\end{remark}

\begin{example}\label{exam:KS8}
Let $\beta=-K_{S_8}$ on $S=S_8$. Then $\beta$ is nef and big but not $0$-very ample because the linear system $|-K_{S_8}|$ has a base point. So Proposition~\ref{local bps g1} does not apply. In fact by Lemma~\ref{lem:Rocco}, this is the only case where $p_a(\beta)=1$ and $\beta$ is not $0$-very ample while there is no line class $l$ such that $\beta. l\le 0$. In this case, we can directly calculate the local BPS number. Since $w=K_{S_8}^2=1$, there are no wall-crossings. Also, $P_{-1}(S, \beta)$ is empty. Hence $M_\beta\simeq P_{1}(S_8, -K_{S_8})$. The moduli space $P_{1}(S_8, -K_{S_8})$ is the space of pairs $(C,p)$ of a point $p$ on $\ptwo$ and a cubic curve $C$ passing through $p$ and the 8 points of the blow-up. Hence it is the total space of the pencil of cubic curves and is isomorphic to $\ptwo$ blown up at 9 base points of the pencil. We see that $P_t(M_\beta)= 1+10t+t^2 $ and $n_\beta=12=\euler(S_8)+1$. 
\end{example}

\subsection{Arithmetic genus 2}\label{sec:localg2}

Now we compute the local BPS invariants for curve classes $\beta\not=-2K_{S_8}$ with arithmetic genus 2. By Proposition~\ref{prop:pullback}, it suffices to consider very ample classes by blowing down all lines $l$ with $\beta.l=0$. The following lemma shows that there are only finitely many such classes. 

\begin{lemma}\label{lem:vamp} 
If $\beta$ is a very ample curve class on a del Pezzo surface $S$ of arithmetic genus $2$, then $\beta+K_S$ is effective with $p_a(\beta+K_S)=0$ and $(-K_S)(\beta+K_S)=2$. Hence $\beta+K_S$ is a conic class as in \eqref{g0d2cl}.
\end{lemma}
\begin{proof}
We have $\beta. (\beta+K_S)= 2p_a(\beta)-2=2$. By \eqref{eq:h0}, $h^0(\beta+K_S)=2$. Therefore $\beta+K_S$ is effective.

Let $\lambda=(-K_S)(\beta+K_S)> 0$. 
We have 
\begin{align*}
p_a(\beta+K_S) &= \frac12 (\beta+K_S)(\beta+K_S+K_S)+1\\
&= p_a(\beta) + K_S(\beta+K_S)\\
&=2-\lambda.
\end{align*}
Thus, $p_a(\beta+K_S)<2$ and $p_a(\beta+K_S)=(\beta+K_S)^2$. Suppose that $p_a(\beta+K_S)=1$. Then $\lambda=(-K_S)(\beta+K_S)=1$ and $(\beta+K_S)^2=1$. 
By applying the Hodge index theorem to the lattice generated by $-K_S$ and $\beta+K_S$, 
we see that this is possible only if $\beta=-2K_{S_8}$. But $-2K_{S_8}$ is not very ample. 

Now suppose $p_a(\beta+K_S)<0$. In this case, $\beta+K_S$ is not nef. 
Since all effective curve classes on $\PP^2$ or $\pone\times\pone$ are nef, 
we have $S=S_r$ with $r\geq 1$. 
If $r\geq 2$, there is a line $l$ such that $(\beta+K_S)l<0$.
Then $\beta.l<(-K_S).l=1$ which contradicts that $\beta$ is very ample. 
If $r=1$ and $\beta=(d; a)$, we have $d>a>0$ by ampleness, 
and $d\geq a+2$ from $p_a(\beta)=2$. 
Then $(\beta+K_S).E_1=a-1\geq 0$ and $(\beta+K_S).(H-E_1)=d-a-2\geq 0$, 
so $\beta+K_S$ is nef, a contradiction. 

Therefore $p_a(\beta+K_S)=0$ and $(-K_S)(\beta+K_S)=2$. 
\end{proof}

\begin{lemma}\label{lem:wcg2}
 Let $\beta$ be a very ample curve class on $S$ of arithmetic genus 2. Then nontrivial wall-crossings for $M^{\delta}_{\beta, n}$ arise if there is a decomposition
 \begin{equation}\label{eq:g2decomp}(1,(\beta,1))=(1,(\beta_1,0)) +(0,(\beta_2,1)),\end{equation}
where $\beta_1$ and $\beta_2$ are one of the following. 
\begin{itemize}
\item[(i)] $p_a(\beta_1)=1$, $p_a(\beta_2)=0$, $-K_S.\beta_1 =w-2$, $-K_S. \beta_2=2$, $\beta_2^2=0$ and $\beta_1. \beta_2=2$. There is a unique such pair $(\beta_1,\beta_2)$. It corresponds to the wall $\delta_0=\frac12 w-1$. 
\item[(ii)] $p_a(\beta_1)=1$, $p_a(\beta_2)=0$, $-K_S.\beta_1 =w-1$, $-K_S. \beta_2=1$, $\beta_2^2=-1$ and $\beta_1. \beta_2=2$. The number of such pairs $(\beta_1,\beta_2)$ is $2e(S)-8$. They correspond to the wall $\delta_0=w-1$.
\end{itemize}
\end{lemma}

\begin{proof}
By the previous lemma, we have 
$(-K_S).\beta=K_S^2+2>2$. 
The list of all very ample classes of arithmetic genus 2 can be obtained by Lemma~\ref{lem:vamp} and Example~\ref{exam:g0w2}. We can check the assertions for each curve classes. For example, if $\beta=(4;2,1,1,1,1)$, then the possible decompositions of $\beta$ on $S_5$ are
\begin{itemize}
\item $ (4;2,1,1,1,1)=(3;1,1,1,1,1)+(1;1,0,0,0,0)$,
\item $(4;2,1,1,1,1)=(3;1,1,1,1,0)+(1;1,0,0,0,1)$ (4 decompositions of this type),
\item $(4;2,1,1,1,1)=(4;2,2,1,1,1)+E_2$ (4 decompositions of this type).
\end{itemize}
The first decomposition is the case (i) of the statement and the remaining two correspond to the case (ii). The other cases can be checked similarly.
\end{proof}

\begin{proposition}\label{prop:ng2}
Let $\beta$ be a curve class on $S=S_r$ of arithmetic genus 2, and assume that $\beta$ is very ample. Then $$P_t(M_\beta)=\displaystyle\frac{1-t^w}{1-t}\left(1+(e(S)-2)t + \left(\binom{e(S)-2}{2}+4\right)t^2+(e(S)-2)t^3+t^4 \right)$$ and $n_{\beta}=(-1)^{w-1} w \left(\binom{\euler(S)}{2}+5\right)$.
\end{proposition}

\begin{proof}
By Proposition~\ref{prop:PT}, $P_t(P_{\beta,1}) = P_t(\PP^{w-1})P_t(\mathrm{Hilb}^2(S))$ and $P_t(P_{\beta,-1}) = P_t(\PP^{w+1})$. We have a wall-crossing for each decomposition in Lemma~\ref{lem:wcg2}. By the similar calculation as before we compute the wall-crossing. 
For the decomposition in Lemma~\ref{lem:wcg2}(1),
\begin{align*}
\Ext^1(  (s, \shO_{\beta_1}), (0, \shO_{\beta_2})   )&\simeq  \CC^3,\\
\Ext^1( (0, \shO_{\beta_2}) , (s, \shO_{\beta_1})  )&\simeq  \CC^2.
\end{align*}
Since $(s, \shO_{\beta_1})\in M^\infty_{{\beta_1},0}\simeq \IP^{w-2}$ and $(0, \shO_{\beta_2}) \in M_{{\beta_2}, 1}\simeq \IP^1$, the correction term for the Poincar\'{e} polynomial in this case is $t^2P_t(\PP^{w-2}) P_t(\pone)$.

For the decomposition in Lemma~\ref{lem:wcg2}(2),
\begin{align*}
\Ext^1(  (s, \shO_{\beta_1}), (0, \shO_{\beta_2})   )&\simeq  \CC^3,\\
\Ext^1( (0, \shO_{\beta_2}) , (s, \shO_{\beta_1})  )&\simeq  \CC^2.
\end{align*}
In this case $(s, \shO_{\beta_1})\in M^\infty_{{\beta_1},0}\simeq \IP^{w-1}$ and $(0, \shO_{\beta_2}) \in M_{{\beta_2}, 1}\simeq \IP^0$. So, the correction term in this case is $(2\euler(S)-8)t^2 P_t(\PP^{w-1})$.

Therefore, we have
\[
P_t(M_\beta) =  P_t(\PP^{w-1})P_t(\mathrm{Hilb}^2(S)) - P_t(\PP^{w+1})- t^2P_t(\PP^{w-2}) P_t(\pone)-(2\euler(S)-8)t^2 P_t(\PP^{w-1}).\]

The Poincar\'{e} polynomial of the Hilbert scheme is well known \cite{hilb}. For the Hilbert scheme of two points, we have 
\[P_t(\mathrm{Hilb}^2(S))=1+(\euler(S)-1)t+ \binom{\euler(S)}{2} t^2 +(\euler(S)-1)t^3 +t^4.\]
Then the result follows from elementary calculations. 
\end{proof}

\begin{remark}
Without the very-ampleness assumption, we can calculate $P_t(M_\beta)$ and $n_\beta$ by using the blowup property. If $\beta$ is nef and big but not very ample, then we may blow down all lines $l$ with $\beta.l=0$. Let $\pi:S\to S'$ be the blowdown. After blowdown, $\pi_*\beta$ is very ample unless $\beta=-2K_{S_8}$, since there are no $(-1)$-curves which do not intersect $\beta$. Therefore we may apply Proposition~\ref{prop:ng2} to calculate $n_{\pi_*\beta}$ on $S'$. Then by Proposition~\ref{prop:pullback}, $ M_\beta\simeq M_{\pi_*\beta}$. Hence if we let $\eta$ be the number of lines $l$ such that $\beta. l=0$ as before, we conclude that
 $$P_t(M_\beta)=\displaystyle\frac{1-t^w}{1-t}\left(1+(e(S)-2-\eta)t + \left(\binom{e(S)-2-\eta}{2}+4\right)t^2+(e(S)-2-\eta)t^3+t^4 \right)$$ and 
\[n_{\beta}=(-1)^{w-1} w \left(\binom{\euler(S)-\eta}{2}+5\right). \]
\end{remark}

\begin{remark} 
For $\PP^1\times \PP^1$, we can check the only very ample classes with arithmetic genus 2 are $2h_1+3h_2$ and $3h_1 +2h_2$. The same calculation works for these classes and we have 
\[P_t(M_\beta)=\displaystyle\frac{1-t^{10}}{1-t}(1+2t +5t^2+2t^3+t^4),\]
which matches with the result of Proposition~\ref{prop:ng2} as $\euler(\pone\times\pone)=4$. 

For these cases, the geometry of $M_\beta$ is studied in \cite{Moon}. We remark that the Poincar\'{e} polynomial obtained in \cite[Corollary 3.8]{Moon} using different birational method agrees with ours. 
\end{remark}

\begin{remark}
Let $\beta=-2K_{S_8}=(6,2^8)$. This curve class is neither very ample nor contracted to a very ample divisor. So, it is not covered by Proposition~\ref{prop:ng2}.

\end{remark}

In conclusion, we have the following formulas for the Poincar\'{e} polynomials and the local BPS invariants.

\begin{Theorem}\label{thm:poincare}
Let $\beta$ be either a line class, a conic class or a nef and big curve class on a del Pezzo surface $S$ of arithmetic genus at most 2. Let $w=(-K_S).\beta$ and let $\eta$ be the number of disjoint lines $l$ such that $\beta. l =0$. Then we have
\begin{itemize}
\item[(i)] if $p_a(\beta)=0$, then $P_t(M_\beta)=\displaystyle\frac{1-t^w}{1-t}$,
\item[(ii)] if $p_a(\beta)=1$ and $\beta\ne -K_{S_8}$, then $P_t(M_\beta)=\displaystyle\frac{1-t^w}{1-t}\left(1+(\euler(S)-2-\eta)t +t^2\right)$,
\item[(iii)] if $\beta=-K_{S_8}$, then $P_t(M_\beta)=1+10t+t^2 $,
\item[(iv)] if $p_a(\beta)=2$ and $\beta\ne -2K_{S_8}$, then $$P_t(M_\beta)=\displaystyle\frac{1-t^w}{1-t}\left(1+(e(S)-2-\eta)t + \left(\binom{e(S)-2-\eta}{2}+4\right)t^2+(e(S)-2-\eta)t^3+t^4 \right).$$
\end{itemize}
\end{Theorem}

\begin{Theorem}\label{thm:localbps}
In the situation as in Theorem~\ref{thm:poincare}, we have 
\begin{itemize}
\item[(i)] if $p_a(\beta)=0$, then $n_\beta= (-1)^{w-1} w$,
\item[(ii)] if $p_a(\beta)=1$ and $\beta\ne -K_{S_8}$, then $n_\beta= (-1)^{w-1} w (e(S)-\eta)$,
\item[(iii)] if $\beta=-K_{S_8}$, then $n_\beta=12$,
\item[(iv)] if $p_a(\beta)=2$ and $\beta\ne -2K_{S_8}$, then $n_\beta= (-1)^{w-1} w \left(\binom{\euler(S)-\eta}{2}+5\right)$.
\end{itemize}
\end{Theorem}

In all cases studied in this paper, we see that $P_t(M_\beta)$ has a factor of $P_t(\PP^{w-1})=\displaystyle\frac{1-t^w}{1-t}$. This may suggest that $M_\beta$ has a projective bundle structure. However, it is not true in general. The stable base locus decomposition of $M_\beta$ when $S=\PP^2$ is studied in \cite{cc2}. It is shown there that $M_\beta$ is not itself a projective bundle but is birational to a projective bundle. We formulate the following conjecture, which we proved for $\beta$ of arithmetic genus at most 2.. 

\begin{Conjecture}\label{conj:div}
Let $\beta$ be either a line class, a conic class or a nef and big curve class on a del Pezzo surface $S$. Let $w=(-K_S).\beta$. Then $P_t(M_\beta)$ has a factor of  $P_t(\PP^{w-1})$ and the quotient $P_t(M_\beta)/P_t(\PP^{w-1})$ is a palindromic polynomial. Consequently, $n_\beta$ is divisible by $w$. 
\end{Conjecture}

This conjecture motivated the theory of log BPS numbers. 

\begin{definition}\label{def:logbps}
We define the log BPS numbers by $m_\beta= (-1)^{w-1}n_\beta/w$. 
\end{definition}

\begin{remark}
In a sequel \cite{logbps}, we give a more geometric approach to the log BPS numbers. We fix a smooth anticanonical divisor $E\in |-K_S|$. The set $E(\beta)$ of points $P$ on $E$ such that there is a curve in class $\beta$ meeting $E$ only at $P$ is a finite set. Roughly speaking, given a point $P\in E(\beta)$, the log BPS number counts the virtual number of rational curves in class $\beta$ which meet $E$ only at $P$ and are smooth at $P$. In \cite{logbps}, we give a precise definition of log BPS numbers using log Gromov-Witten theory and conjecture that it is constant along points $P\in E(\beta)$. When $P\in E(\beta)$ is \emph{$\beta$-primitive}, which means that there are only reduced irreducible rational curves in class $\beta$ meeting $E$ only at $P$, this is an actual count of curves. In this case, we show that the log BPS numbers of Definition~\ref{def:logbps} agree with the number of such rational curves when $\beta$ has arithmetic genus at most 2. 
\end{remark}

\subsection{Refined BPS indices and higher genus BPS invariants}\label{sec:refined}

In \cite{GV2}, physical reasoning was used to assert an $sl_2\times sl_2$-representation on the cohomology $H^*(M_\beta)$ of the moduli space $M_\beta$ which refines the Gopakumar-Vafa invariants. The left and the right $sl_2$-actions are given by the Lefschetz actions from the maps $M_\beta\to |\beta| \to pt$ respectively. A mathematical proposal for an $sl_2\times sl_2$-representation was given in \cite{KL}.  While a counterexample to this proposal was found in \cite{MT}, the problem does not occur for smooth moduli spaces.  We therefore can and will use the proposal of \cite{KL} as a precise mathematical definition.

A computational algorithm for such $sl_2\times sl_2$-representations based on conjectures from physics was developed in \cite{kkv} and generalized in \cite{ckk} using the refined PT invariants. Adapting the notations in \cite{ckk} we let $[\frac{k}{2}]$ denote the irreducible $sl_2$-representation of dimension $k+1$. Then we may write $H^*(M_\beta)= \sum_{j_L, j_R} N_{j_L, j_R}^\beta [j_L, j_R]$ as an $sl_2\times sl_2$-representation, where $j_L, j_R\in \frac12 \mathbb{Z}$. The multiplicity $N_{j_L, j_R}^\beta $ is called the \emph{refined BPS index}. There is a conjectural product formula for the generating function of the refined PT invariant in terms of the refined BPS indices, see \cite[\S 8]{ckk}.

The computation of the previous sections can be applied to the calculation algorithms in \cite{ckk} of the refined BPS indices. We present the results here omitting the details. In the following, we assume $[\frac{k}{2}]=0$ for $k<0$.

Let $r=e(S)-3$ as before. (For $S=\PP^1\times \PP^1$, $r=1$.)
\begin{itemize}
\item[(i)] If $p_a(\beta)=0$, then $H^*(M_\beta)= [0, \frac{w-1}{2}]$.
\item[(ii)] If $p_a(\beta)=1$ and $\beta\ne -K_{S_8}$, then $H^*(M_\beta)= [\frac12, \frac{w}{2}] + (r-\eta)[0, \frac{w-1}{2}] + [0, \frac{w-3}{2}]$.
\item[(iii)] If $\beta=-K_{S_8}$, then $H^*(M_\beta)=[\frac12,\frac12]+ 8[0,0]$. 
\item[(iv)] If $p_a(\beta)=2$ and $\beta\ne -2K_{S_8}$, then $H^*(M_\beta)=[1, \frac{w+1}{2}] +(r-\eta)[\frac12, \frac{w}{2}] +[\frac12, \frac{w-2}{2}] + \left(\binom{r-\eta}{2} +2\right) [0,\frac{w-1}{2}]+ (r-\eta)[0, \frac{w-3}{2}]+[0, \frac{w-5}{2}]$.
\end{itemize}

These results are consistent with the refined BPS indices obtained by mirror symmetry in \cite[\S 5]{HKP}. We remark that $ N_{j_L, j_R}^d$ in \cite[\S 5]{HKP} is $\sum_{(-K_S).\beta=d}  N_{j_L, j_R}^\beta$.

Upon restricting to the representation $(H^*(M_\beta))_\Delta$ of the diagonal $(sl_2)_\Delta \subset sl_2\times sl_2$, we recover the cohomology of $M_\beta$. By simple computation, we see that 
 
\begin{itemize}
\item[(i)] if $p_a(\beta)=0$, then $(H^*(M_\beta))_\Delta= [\frac{w-1}{2}]$,
\item[(ii)] if $p_a(\beta)=1$ and $\beta\ne -K_{S_8}$, then $(H^*(M_\beta))_\Delta= [\frac{w-1}{2}]([1]+(r-\eta)[0] )$,
\item[(iii)] if $\beta=-K_{S_8}$, then $(H^*(M_\beta))_\Delta=[1]+ 9[0]$,
\item[(iv)] if $p_a(\beta)=2$ and $\beta\ne -2K_{S_8}$, then $(H^*(M_\beta))_\Delta=[\frac{w-1}{2}]\left([2] +(r-\eta) [1] +\left(\binom{r-\eta}{2} +3\right)[0] \right)$.
\end{itemize}

In each case, $(H^*(M_\beta))_\Delta$ is divisible by $[\frac{w-1}2]$, consistent with Conjecture~\ref{conj:div}, as $[\frac{w-1}2]$ is the Lefschetz representation of $\PP^{w-1}$.

\end{document}